\documentclass[a4paper,12pt]{amsart}
\usepackage{amsmath}
\usepackage{amssymb}
\usepackage{amscd}
\usepackage{mathrsfs}
\usepackage{url}
\usepackage[all]{xy}
\usepackage{graphicx}
\usepackage{tikz}
\everymath{\displaystyle}
\theoremstyle{plain} 
\newtheorem{theorem}{\indent\sc Theorem}[section]
\newtheorem{lemma}[theorem]{\indent\sc Lemma}
\newtheorem{corollary}[theorem]{\indent\sc Corollary}
\newtheorem{proposition}[theorem]{\indent\sc Proposition}

\theoremstyle{definition} 

\newtheorem{remark}[theorem]{\indent\sc Remark}
\newtheorem{example}[theorem]{\indent\sc Example}

\newcommand{\del}{\partial}

\newcommand{\ve}{\varepsilon}

\begin{document}

\pagestyle{plain}
\thispagestyle{plain}

\title[Rigidity of Pseudoconvex Neighborhoods and Ohsawa--Takegoshi Extension]
{Rigidity of Pseudoconvex Neighborhoods and Ohsawa--Takegoshi Extension}

\author[Takayuki Koike and Masakazu Takakura]{Takayuki KOIKE$^{1}$ and Masakazu Takakura$^{2}$}
\address{ 
$^{1}$ Division of Mathematics \\
Institute of Pure and Applied Sciences \\
University of Tsukuba \\
1-1-1 Tennodai, Tsukuba \\
Ibaraki 305-8571 \\
Japan 
}
\email{tkoike@math.tsukuba.ac.jp}
\address{
$^{2}$ Department of Mathematical Sciences \\
Tokyo Metropolitan University \\
1-1 Minamiosawa, Hachioji-shi \\
Tokyo 192-0397, Japan
}
\email{takakura-masakazu@ed.tmu.ac.jp}
\renewcommand{\thefootnote}{\fnsymbol{footnote}}
\renewcommand{\thefootnote}{\arabic{footnote}}

\begin{abstract} 
Let $Y$ be a holomorphic section of a holomorphic submersion
with one-dimensional fibers over a connected compact complex manifold.
Assuming that its normal bundle is unitary flat, we prove that $Y$
admits a fundamental system of locally pseudoconvex neighborhoods
if and only if it admits a holomorphic tubular neighborhood.
The proof uses the sharp Ohsawa--Takegoshi extension theorem.
As an application, we rule out such neighborhood systems
in Ueda's example of class $(\gamma)$.
\end{abstract}

\maketitle

\section{Introduction}\label{section:introduction}

In this paper, we investigate neighborhoods of a holomorphic section of a
holomorphic submersion with one-dimensional fibers over a connected compact
complex manifold. Throughout the paper, we let

\begin{itemize}
\item $Z$ be a connected compact complex manifold of dimension $n\geq 0$;
\item $X$ be a complex manifold of dimension $n+1$;
\item $p\colon X\to Z$ be a holomorphic surjective submersion;
\item $s\colon Z\to X$ be a holomorphic section of $p$ (namely, $p\circ s=\mathrm{id}_Z$); and
\item $Y:=s(Z)\subset X$. 
\end{itemize}
Then $Y$ is a non-singular compact complex hypersurface of $X$, naturally
biholomorphic to $Z$ via $s$ and $p|_Y$.

We are interested in the following question:
\emph{When does $Y$ admit a fundamental system
$\{\Omega_\nu\}_{\nu\in\mathbb{Z}_{\geq 0}}$ of locally pseudoconvex neighborhoods?}
In this paper, we study this question under the assumption that the normal
bundle $N_{Y/X}$ of $Y$ is \emph{unitary flat}: for a suitable choice of
holomorphic local frames, its transition functions are locally constant
and take values in $\mathrm{U}(1):=\{t\in\mathbb{C}\mid |t|=1\}$.
Declaring these local frames to have norm $1$ defines a Hermitian metric
whose Chern curvature vanishes identically (a
\emph{flat metric} on $N_{Y/X}$).
In the above setting of a holomorphic surjective submersion
$p\colon X\to Z$ with a holomorphic section $s\colon Z\to X$,
the existence of a fundamental system of locally pseudoconvex
neighborhoods of $Y$ implies that the conormal bundle $N_{Y/X}^*$ is
pseudoeffective (see Lemma~\ref{lem:motomoto_Ast}).
The unitary flat case may therefore be viewed as a borderline case
for the existence of such a neighborhood system.

Basic examples are provided by the zero sections of unit disc bundles
in unitary flat line bundles: shrinking the fiber discs gives locally
pseudoconvex neighborhoods. 
On the other hand, the neighborhood geometry of a compact hypersurface
need not be so simple in general.
In \cite[\S 5.4]{U}, Ueda constructed an example $p\colon X\to Z$ with a
holomorphic section $Y=s(Z)$ with unitary flat normal bundle and $n=1$ such that, for every positive integer
$m_0$ and every sufficiently small neighborhood $\Omega$ of $Y$ in $X$,
there exist an integer $m>m_0$ and a compact non-singular complex curve
$\Gamma\subset\Omega\setminus Y$ homologous to $mY$ in $X$.
For this example and related examples, the first author and Ogawa obtained
results on plurisubharmonic functions near $Y$ \cite{K,KO}, which suggest
restrictions on the existence of pseudoconcave neighborhoods of $Y$.
These results alone, however, do not determine whether $Y$ admits a
fundamental system of locally pseudoconvex neighborhoods.
We prove the following theorem, which implies that no such system exists
in Ueda's example (see Corollary~\ref{cor:main}).

\begin{theorem}\label{thm:main1}
Let $p\colon X\to Z$ be a holomorphic surjective submersion from a complex
manifold $X$ to a connected compact complex manifold, and let $Y$ be a
section of $p$. Assume that $N_{Y/X}$ is a unitary flat line bundle. Then the following
are equivalent:\\
$(i)$ There exists a fundamental system $\{\Omega_\nu\}_{\nu\in\mathbb{Z}_{\geq 0}}$ of neighborhoods of $Y$ in $X$ such that each $\Omega_\nu$ is locally pseudoconvex. \\
$(ii)$ There exists a fundamental system $\{\Omega_\nu\}_{\nu\in\mathbb{Z}_{\geq 0}}$ of neighborhoods of $Y$ in $X$ such that each $\Omega_\nu$ is locally pseudoconvex and each fiber of $p|_{\Omega_\nu}$ is connected and simply connected. \\ 
$(iii)$ There exists a holomorphic tubular neighborhood of $Y$ in $X$. 
\end{theorem}

In assertion $(iii)$ of Theorem~\ref{thm:main1}, a \emph{holomorphic tubular
neighborhood} of $Y$ in $X$ means an open neighborhood $\Omega\subset X$ of
$Y$ for which there exist a neighborhood $\Omega'$ of the zero section in
the total space of $N_{Y/X}$ and a biholomorphism
\begin{equation}\label{eq:holomorphictubularneighborhoodnodefnodoukei}
\Phi\colon \Omega \to \Omega'
\end{equation}
such that $\Phi(Y)$ is the zero section and $(\Phi|_Y)^{-1}$ is the
restriction of the projection $N_{Y/X}\to Y$ to the zero section,
with $Y$ and $Z$ identified by $p|_Y$.

More precisely, the relation between assertions $(ii)$ and $(iii)$ is
expressed by the following theorem.
\begin{theorem}\label{thm:main2}
Let $p\colon X\to Z$ be as in Theorem~\ref{thm:main1}, and let $Y$ be a
section of $p$. Assume that $N_{Y/X}$ is unitary flat.
Let $\{\Omega_\nu\}_{\nu\in\mathbb{Z}_{\geq 0}}$ be a fundamental system
of neighborhoods of $Y$ as in assertion $(ii)$ of Theorem~\ref{thm:main1},
with $\Omega_{\nu+1}\subset\Omega_\nu$ for every
$\nu\in\mathbb{Z}_{\geq 0}$.
Then, for a sufficiently large $\nu$, there exists $\Phi=\Phi_\nu\colon\Omega\to N_{Y/X}$ is as in
\eqref{eq:holomorphictubularneighborhoodnodefnodoukei} for $\Omega=\Omega_\nu$ such that 
\[
\Phi(\Omega_\nu) = \{\xi \in N_{Y/X} \mid |\xi|_h < 1\}
\]
holds for a flat metric $h$ on $N_{Y/X}$ and that $\Phi$ is fiber preserving:
$\pi\circ\Phi=s\circ p$, where $\pi\colon N_{Y/X}\to Y$ is the projection.
In particular, $\Omega_\nu$ is a holomorphic tubular
neighborhood of $Y$ for a sufficiently large $\nu$. 
\end{theorem}

Our results are related to rigidity questions for holomorphic families of
Riemann surfaces. Nishino's rigidity theorem \cite{N} states that a
holomorphic submersion from a Stein surface onto the unit disc, all of
whose fibers are biholomorphic to $\mathbb{C}$, is holomorphically trivial.
Pseudoconvex variations of planar Riemann surfaces have also been studied
in detail using boundary regularity and related assumptions; see, for
example, \cite{HMY,Z}.
One advantage of our approach, based on the Ohsawa--Takegoshi extension
theorem, is that it yields rigidity results without imposing any such
boundary regularity assumptions (see Remark~\ref{rmk:newnew}).

The paper is organized as follows.
Section~2 recalls the Bergman kernel, sharp $L^2$ extension estimates,
and the variation of Bergman kernels. It also discusses the conormal bundle and the choice of
local coordinates. Section~3 proves the main theorems.
Section~\ref{section:localautsusp} treats suspension neighborhoods of
elliptic curves and the application to Ueda's example.
\vskip3mm
We use the following notation and conventions:
\begin{itemize}
\item $\mathrm{U}(1):=\{t\in \mathbb{C}\mid |t|=1\}$; 
\item $\mathbb{C}^*:=\mathbb{C}\setminus\{0\}$; 
\item Local coordinates on $Z$ are denoted by $z=(z^1,z^2,\dots,z^n)$;
\item A fiber coordinate for $p$ is denoted by $w$, so that local coordinates on $X$ are $(z,w)=(z^1,z^2,\dots,z^n,w)$;
\item $\Delta:=\{w\in\mathbb{C}\mid |w|<1\}$ denotes the unit disc; 
\item $\mathbb{B}^n$ denotes the unit ball in $\mathbb{C}^n$; 
\item $\sigma_n$ denotes the Euclidean volume of $\mathbb{B}^n$;
\item $X_z:=p^{-1}(z)$; 
\item For an open neighborhood $\Omega$ of $Y$ in $X$ and $z\in Z$, $\Omega_z$ denotes the fiber $(p|_\Omega)^{-1}(z)=\Omega\cap X_z$;
\item For an open neighborhood $\Omega$ of $Y$ in $X$ and $z\in Z$,
$K_{\Omega_z}$ denotes the Bergman kernel of $\Omega_z$; the definition
and related notation are given in \S\ref{subsection:bergmanker_optimalOT};
\item For a complex manifold $W$, $T_W$ and $T_W^*$ denote its holomorphic
tangent and cotangent bundles, respectively;
\item $N_{Y/X}$ denotes the holomorphic normal bundle of $Y$ in $X$, namely $N_{Y/X}:=T_X|_Y/T_Y$; 
\item $N_{Y/X}^*$ denotes the holomorphic conormal bundle of $Y$ in $X$.
\end{itemize}

\vskip3mm
{\bf Acknowledgment. } 
The first author is supported by the Grant-in-Aid for Scientific Research C (23K03119) from JSPS. 
The second author is supported by JST SPRING, Japan Grant Number JPMJSP2156.

\vskip3mm
\noindent{\bf AI usage statement. }
ChatGPT (OpenAI) was used for language editing, including
the detection and correction of minor typographical errors.
All mathematical ideas and arguments, as well as the formulation of the
results, are entirely due to the authors.

\section{Preliminaries}

\subsection{Variation of Bergman kernels and sharp Ohsawa--Takegoshi extension}\label{subsection:bergmanker_optimalOT}

Although the following discussion also applies in higher dimensions,
we restrict ourselves to one complex variable for simplicity.
For details, see also \cite[Chapter~4]{O}.

In the discussion below, bounded open sets need not be connected.
For an extremal problem at a specified point, one may work on the component
containing that point and extend by zero to the other components;
sharp extension estimates likewise apply componentwise.
Thus the statements below also apply to the possibly disconnected fibers
$\Omega_z$ considered later.

Let $\Omega\subset\mathbb{C}$ be a bounded domain.
We denote by
$\Gamma(\Omega,\mathcal{O}_{\Omega}(T_{\Omega}^*))$
the space of holomorphic sections of the holomorphic
cotangent bundle $T_{\Omega}^*$, that is,
the space of holomorphic $1$-forms on $\Omega$.
The Bergman space of $\Omega$ is the Hilbert space
\[
A^2(\Omega)
:=
\left\{
\eta\in\Gamma(\Omega,\mathcal{O}_{\Omega}(T_{\Omega}^*))
\;\middle|\;
\|\eta\|_{\Omega}^2
:=
\frac{\sqrt{-1}}{2}
\int_{\Omega}\eta\wedge\overline{\eta}
<\infty
\right\}.
\]
Let $\{\eta_\nu\}_{\nu=0}^\infty$ be an orthonormal
basis of $A^2(\Omega)$. The kernel
\begin{equation}\label{eq:def_bergmanker_sumCONB}
K_{\Omega}(w,w')
=
\sum_{\nu=0}^\infty
\eta_\nu(w)\otimes\overline{\eta_\nu(w')},
\qquad w,w'\in\Omega,
\end{equation}
is called the Bergman kernel of $\Omega$.
It is independent of the choice of orthonormal basis.

For example, for the unit disk $\Delta$, using the
orthonormal basis
\[
\eta_\nu(w)
=
\sqrt{\frac{\nu+1}{\pi}}\,w^\nu\,dw,
\qquad \nu\in\mathbb{Z}_{\geq 0},
\]
we obtain
\begin{equation}\label{eq:def_bergmanker_sumCONB_forunitdisc}
K_\Delta(w,w')
=
\frac{dw\otimes d\overline{w'}}
{\pi(1-w\overline{w'})^2}.
\end{equation}

In what follows, we write
\[
K_{\Omega}(w,w')
=
k_{\Omega}(w,w')\,dw\otimes d\overline{w'},
\qquad
k_{\Omega}(w):=k_{\Omega}(w,w).
\]
The Bergman kernel satisfies the reproducing property
\begin{equation}\label{eq:bergman_reproducing}
f(w')
=
\frac{\sqrt{-1}}{2}
\int_{\Omega}
f(w)\,\overline{k_{\Omega}(w,w')}\,
dw\wedge d\overline{w}
\end{equation}
for every $f(w)\,dw\in A^2(\Omega)$ and $w'\in\Omega$.
This follows from \eqref{eq:def_bergmanker_sumCONB}
and the orthonormality of $\{\eta_\nu\}_{\nu=0}^\infty$.

It also follows from \eqref{eq:def_bergmanker_sumCONB}
that, for every $w_0\in\Omega$,
the function $w\mapsto k_{\Omega}(w,w_0)$
is holomorphic on $\Omega$.
Moreover, $k_{\Omega}(w_0)>0$ for every $w_0\in\Omega$.
Indeed, since $\Omega$ is bounded,
we can choose an orthonormal basis
$\{\eta_\nu\}_{\nu=0}^\infty$ with
\[
\eta_0
=
\frac{dw}{\|dw\|_{\Omega}}.
\]
Then \eqref{eq:def_bergmanker_sumCONB} gives
\[
k_{\Omega}(w_0)
\geq
\frac{1}{\|dw\|_{\Omega}^2}>0.
\]

We recall the following extremal property
of the Bergman kernel.

\begin{lemma}\label{lem:b-kerbasic}
Let $\Omega\subset\mathbb{C}$ be a bounded domain.
Then, for every $w_0\in\Omega$ and $a\in\mathbb{C}$,
\[
\frac{|a|^2}{k_{\Omega}(w_0)}
=
\min\left\{
\|f(w)\,dw\|_{\Omega}^2
\;\middle|\;
f(w)\,dw\in A^2(\Omega),\ f(w_0)=a
\right\}.
\]
Moreover, the minimum is attained uniquely by
\[
a\,\frac{k_{\Omega}(w,w_0)}
{k_{\Omega}(w_0)}\,dw.
\]
\end{lemma}

\begin{proof}
By the reproducing property
\eqref{eq:bergman_reproducing}
and the Cauchy--Schwarz inequality, we have
\[
|f(w_0)|^2
\leq
k_{\Omega}(w_0)
\|f(w)\,dw\|_{\Omega}^2
\]
for every $f(w)\,dw\in A^2(\Omega)$.
Hence, if $f(w_0)=a$, then
\[
\|f(w)\,dw\|_{\Omega}^2
\geq
\frac{|a|^2}{k_{\Omega}(w_0)}.
\]
On the other hand, the form
\[
f(w)\,dw
=
a\,\frac{k_{\Omega}(w,w_0)}
{k_{\Omega}(w_0)}\,dw
\]
belongs to $A^2(\Omega)$, satisfies $f(w_0)=a$,
and has squared norm $|a|^2/k_{\Omega}(w_0)$.
Thus it attains the minimum.

Finally, equality in the Cauchy--Schwarz inequality
implies that every minimizer is a constant multiple of
$k_{\Omega}(w,w_0)\,dw$.
The condition $f(w_0)=a$ determines this constant
uniquely.
\end{proof}

\begin{example}\label{ex:bergman-simply-connected}
For the unit disc $\Delta$,
\eqref{eq:def_bergmanker_sumCONB_forunitdisc} gives
$k_\Delta(w,0)=k_\Delta(0)=1/\pi$.
Thus, among the forms $f(w)\,dw\in A^2(\Delta)$ with $f(0)=a$,
the unique minimizer is $a\,dw$, whose squared norm is $\pi|a|^2$,
as in Lemma~\ref{lem:b-kerbasic}.
More generally, let $\Omega\subset\mathbb{C}$ be a bounded simply
connected domain containing $0$.
Take a Riemann map $\varphi\colon\Omega\to\Delta$ normalized by
$\varphi(0)=0$ and $\varphi'(0)>0$.
By the transformation law of the Bergman kernel,
\[
k_\Omega(w,0)\,dw
=\frac{\varphi'(0)}{\pi}\,d\varphi(w),
\qquad
k_\Omega(0)=\frac{\varphi'(0)^2}{\pi}.
\]
Hence the minimizer in Lemma~\ref{lem:b-kerbasic} for $w_0=0$ is
\[
a\,\frac{\varphi'(w)}{\varphi'(0)}\,dw,
\]
with squared norm $\pi|a|^2/\varphi'(0)^2$.
In particular, if $k_\Omega(0)=1$, then
\[
\int_0^w k_\Omega(\zeta,0)\,d\zeta
=\frac{\varphi(w)}{\sqrt{\pi}},
\]
which maps $\Omega$ biholomorphically onto
$\{\xi\in\mathbb{C}\mid |\xi|<1/\sqrt{\pi}\}$.
\end{example}

Let $\Omega\subset\mathbb{C}^{n+1}
=\mathbb{C}^n\times\mathbb{C}$
be a bounded pseudoconvex domain satisfying
\[
\Omega\subset
\left\{
(z,w)\in\mathbb{C}^{n+1}
\mid |z|<1
\right\},
\]
where $z=(z^1,z^2,\dots,z^n)$ denotes the coordinates
on $\mathbb{C}^n$ and $w$ the coordinate on $\mathbb{C}$.
Set
\[
\Omega_t:=\Omega\cap\{z=t\},
\]
and assume that $\Omega_0$ is nonempty.
We shall use the following special case of the
sharp Ohsawa--Takegoshi $L^2$ extension theorem
\cite{Be,Bl,GZ}.

\begin{theorem}[Sharp Ohsawa--Takegoshi $L^2$ extension theorem]
\label{thm:sharp_OT}
For every $f(w)\,dw\in A^2(\Omega_0)$,
there exists a holomorphic function
$F\colon\Omega\to\mathbb{C}$
such that $F(0,w)=f(w)$ on $\Omega_0$ and
\[
\frac{1}{\sigma_n}
\int_{\Omega}|F(z,w)|^2\,d\lambda(z,w)
\leq
\|f(w)\,dw\|_{\Omega_0}^2.
\]
Here $d\lambda$ denotes the Lebesgue measure,
and $\sigma_n=\pi^n/n!$ is the Euclidean volume
of the unit ball in $\mathbb{C}^n$.
\end{theorem}

For further details, see, for example,
\cite[pp.~1149--1151]{GZ},
\cite[Theorem~3.1]{BL},
and \cite[\S~4.4.4]{O}.

We recall the logarithmic plurisubharmonicity of the diagonal Bergman
kernel, established by Maitani and Yamaguchi \cite{MY} for Riemann surfaces
and developed further by Berndtsson \cite{Be} and Guan and Zhou \cite{GZ}.
We use the proof of Guan and Zhou
\cite[Section~3.5, pp.~1149--1151]{GZ}, based on their sharp
$L^2$ extension theorem.

\begin{theorem}\label{thm:bergman-log-psh}
The function
\[
z \longmapsto\log k_{\Omega_z}(0)
\]
is plurisubharmonic on $\{z\mid (z,0)\in\Omega\}$.
\end{theorem}
\begin{proof}

Fix a point of $\Omega$.
After changing the coordinates,
we may assume that this point is $(0,0)$ and that
$\overline{\mathbb B^n}\times\{0\}\subset\Omega$.
Write
\[
\sigma_n:=\operatorname{Vol}(\mathbb B^n)=\frac{\pi^n}{n!},
\qquad
d\lambda_w:=\frac{\sqrt{-1}}{2}\,dw\wedge d\overline w,
\qquad
d\lambda=d\lambda_z\,d\lambda_w,
\]
where $d\lambda_z$ and $d\lambda$ denote the Lebesgue measures
on $\mathbb C^n$ and $\mathbb C^{n+1}$, respectively.

Apply Theorem~\ref{thm:sharp_OT} to the holomorphic $1$-form
$k_{\Omega_0}(w,0)\,dw$ on $\Omega_0$.
There exists a holomorphic function
$F$ on $p^{-1}(\mathbb B^n)$ such that
\[
F(0,w)=k_{\Omega_0}(w,0)
\]
and
\begin{equation}\label{eq:bergman-sharp-extension}
\frac{1}{\sigma_n}
\int_{p^{-1}(\mathbb B^n)}|F|^2\,d\lambda
\leq
\bigl\|k_{\Omega_0}(w,0)\,dw\bigr\|_{\Omega_0}^2
=
k_{\Omega_0}(0).
\end{equation}

For almost every $z\in\mathbb B^n$, the form $F(z,w)\,dw$
is square integrable on $\Omega_z$.
By Lemma~\ref{lem:b-kerbasic}, we have
\begin{equation}\label{eq:bergman-fiber-extremal}
\frac{|F(z,0)|^2}{k_{\Omega_z}(0)}
\leq
\int_{\Omega_z}|F(z,w)|^2\,d\lambda_w.
\end{equation}
Integrating \eqref{eq:bergman-fiber-extremal} with respect to $z$
and applying Fubini's theorem, we obtain
\begin{equation}\label{eq:bergman-integrated-extremal}
\int_{\mathbb B^n}
\frac{|F(z,0)|^2}{k_{\Omega_z}(0)}\,d\lambda_z
\leq
\int_{p^{-1}(\mathbb B^n)}|F|^2\,d\lambda.
\end{equation}
Combining \eqref{eq:bergman-sharp-extension} and
\eqref{eq:bergman-integrated-extremal} gives
\begin{equation}\label{eq:bergman-extension-chain}
\frac{1}{\sigma_n}
\int_{\mathbb B^n}
\frac{|F(z,0)|^2}{k_{\Omega_z}(0)}\,d\lambda_z
\leq
\frac{1}{\sigma_n}
\int_{p^{-1}(\mathbb B^n)}|F|^2\,d\lambda
\leq k_{\Omega_0}(0).
\end{equation}

The boundedness of $\Omega$ and the inclusion
$\overline{\mathbb B^n}\times\{0\}\subset\Omega$ imply that
$k_{\Omega_z}(0)$ is bounded above and bounded away from zero
on $\mathbb B^n$.
It follows from \eqref{eq:bergman-extension-chain} that
$F(\,\cdot\,,0)\in L^2(\mathbb B^n)$.
Moreover, $F(0,0)=k_{\Omega_0}(0)>0$.
Hence the logarithms occurring below are integrable.

By Jensen's inequality for the concave function $\log$,
\eqref{eq:bergman-extension-chain} yields
\begin{equation}\label{eq:bergman-jensen}
\begin{aligned}
&\frac{1}{\sigma_n}
\int_{\mathbb B^n}
\bigl(\log|F(z,0)|^2-\log k_{\Omega_z}(0)\bigr)
\,d\lambda_z
\\
&\qquad\leq
\log\left(
\frac{1}{\sigma_n}
\int_{\mathbb B^n}
\frac{|F(z,0)|^2}{k_{\Omega_z}(0)}\,d\lambda_z
\right)
\leq \log k_{\Omega_0}(0).
\end{aligned}
\end{equation}
On the other hand, $\log|F(z,0)|^2$ is plurisubharmonic.
Its submean inequality gives
\begin{equation}\label{eq:bergman-holomorphic-submean}
2\log k_{\Omega_0}(0)
=
\log|F(0,0)|^2
\leq
\frac{1}{\sigma_n}
\int_{\mathbb B^n}\log|F(z,0)|^2\,d\lambda_z.
\end{equation}
Combining \eqref{eq:bergman-jensen} and
\eqref{eq:bergman-holomorphic-submean}, we conclude that
\[
2\log k_{\Omega_0}(0)
-
\frac{1}{\sigma_n}
\int_{\mathbb B^n}\log k_{\Omega_z}(0)\,d\lambda_z
\leq \log k_{\Omega_0}(0).
\]
Therefore,
\begin{equation}\label{eq:bergman-submean}
\log k_{\Omega_0}(0)
\leq
\frac{1}{\sigma_n}
\int_{\mathbb B^n}\log k_{\Omega_z}(0)\,d\lambda_z.
\end{equation}

The same argument applies at every point and on every sufficiently
small ball.
Applying the case $n=1$ to the restriction of the family to each
affine complex line in the base,
we find that $z\mapsto\log k_{\Omega_z}(0)$ is plurisubharmonic.
\end{proof}

When the function in Theorem~\ref{thm:bergman-log-psh} is pluriharmonic,
we also obtain the following.

\begin{proposition}\label{prop:kyoutuketasita}
If the function
\[
z \longmapsto\log k_{\Omega_z}(0)
\]
is pluriharmonic on $\{z\mid (z,0)\in\Omega\}$, then
\[
(z, w) \longmapsto \frac{k_{\Omega_z}(w, 0)}{k_{\Omega_z}(0)}
\]
is holomorphic. In particular, if $\log k_{\Omega_z}(0)\equiv 0$, then
$(z, w) \mapsto k_{\Omega_z}(w, 0)$ is holomorphic.
\end{proposition}

\begin{proof}
If $z\mapsto\log k_{\Omega_z}(0)$ is pluriharmonic, then
\eqref{eq:bergman-submean} in the proof of
Theorem~\ref{thm:bergman-log-psh} is an equality.
It follows that \eqref{eq:bergman-fiber-extremal} is also an equality:
\[
\frac{|F(z,0)|^2}{k_{\Omega_z}(0)} = \int_{\Omega_z}|F(z,w)|^2\,d\lambda_w
\]
Applying Lemma~\ref{lem:b-kerbasic} on each $\Omega_z$, we obtain
\[
\frac{k_{\Omega_z}(w, 0)}{k_{\Omega_z}(0)} = \frac{F(z, w)}{F(z, 0)}
\]
at each $(z,w)$, which proves the assertion.
\end{proof}

\begin{remark}\label{rmk:newnew}
Let $\Omega\subset \mathbb{B}^n\times\mathbb{C}$ be a pseudoconvex domain
such that $(z,0)\in\Omega$ for every $z\in\mathbb{B}^n$ and
$k_{\Omega_z}(0)\equiv 1$.
If the function $z\mapsto\log k_{\Omega_z}(0)$ is pluriharmonic, this
normalization can be achieved by a holomorphic change of the fiber
coordinate; based on the transition formula \eqref{eq:b-ker_zahyouhenkan} below, consider the exponential of a holomorphic function whose real part coincides with $z\mapsto\log k_{\Omega_z}(0)$.
If, in addition, each fiber $\Omega_z$ is connected and simply connected,
then the projection $\Omega\to\mathbb{B}^n$ is holomorphically locally
trivial.
Indeed, Proposition~\ref{prop:kyoutuketasita} implies that
$(z,w)\mapsto k_{\Omega_z}(w,0)$ is holomorphic.
Integrating $k_{\Omega_z}(w,0)\,dw$ along paths in the fibers starting at
$0$ therefore defines a new holomorphic fiber coordinate, which gives a
local trivialization by Example~\ref{ex:bergman-simply-connected}.
The proof of Theorem~\ref{thm:main2} is based on this observation.
This argument requires neither smoothness of the boundaries of $\Omega$
and its fibers nor any regularity assumptions on the variation of the
fiber boundaries.
In this respect, it differs from approaches using boundary regularity in
earlier work such as \cite{HMY,Z}.
\end{remark}

\subsection{Pseudoeffectivity of the conormal bundle}\label{subsection:conormalPE}

Assume condition $(i)$ of Theorem~\ref{thm:main1}.
Since $p$ is a holomorphic submersion, near each point of $Y$ we can choose
holomorphic coordinates $(z,w)$ in which $p(z,w)=z$.
For sufficiently large $\nu$, the fibers of $\Omega_\nu$ over a fixed small
coordinate ball in $Z$ can then be viewed as bounded open subsets of
$\mathbb{C}_w$; they are domains when $p|_{\Omega_\nu}$ has connected fibers.
In these coordinates, local pseudoconvexity is considered in
$\mathbb{C}^{n+1}$.

We work locally in this sense. Let $U\cong\mathbb{B}^n$ be a coordinate
ball about a point of $Z$, and use this identification to write
$U=\mathbb{B}^n$. Choose $\nu$ sufficiently large that
\[
\Omega_\nu\cap p^{-1}(\mathbb{B}^n) \subset \mathbb{B}^n\times \mathbb{C}_w
\]
can be regarded as an inclusion in the chosen coordinates.
Since $Y$ is a section, we may change only the fiber coordinate $w$ so that
\[
Y\cap p^{-1}(U) =\{(z, w)\in \mathbb{B}^n\times \mathbb{C}_w \mid w=0\}
\]

\begin{lemma}\label{lem:motomoto_Ast}
Assume condition $(i)$ of Theorem~\ref{thm:main1}, without assuming that the
conormal bundle is unitary flat.
Then $N_{Y/X}^*$ admits a singular Hermitian metric with local weights
$\log k_{\Omega_{\nu,z}}(0)$.
These weights are locally bounded and plurisubharmonic.
In particular, $N_{Y/X}^*$ is pseudoeffective.
\end{lemma}

\begin{proof}
In two fiber coordinates $w$ and $w'$, write
$K_{\Omega_z}(w',w')=k'_{\Omega_z}(w',w')\,dw'\otimes d\overline{w'}$.
Since
\[
k'_{\Omega_z}(w', w')\,dw'\otimes d\overline{w'} = k_{\Omega_z}(w, w)\,dw\otimes d\overline{w}
\]
we have
\begin{equation}\label{eq:b-ker_zahyouhenkan}
k'_{\Omega_z}(w', w')=k_{\Omega_z}(w, w)\cdot \left|\frac{dw}{d w'}\right|^2
\end{equation}
This is the required transformation law for the local weights of a
Hermitian metric on the conormal bundle.
Their plurisubharmonicity follows from
Theorem~\ref{thm:bergman-log-psh}.
\end{proof}

\begin{remark}\label{rmk:daijinarmk}
Since $Z$ is compact, it has a finite covering $\{U_j\}_{j\in J}$ by
coordinate balls, with $\#J<\infty$ and $U_j\cong\mathbb{B}^n$.
Write $z_j=(z_j^1,z_j^2,\dots,z_j^n)$ for local coordinates defined on a
neighborhood of $\overline{U_j}$.
Under condition $(i)$ of Theorem~\ref{thm:main1}, the preceding discussion
shows that, for sufficiently large $\nu$, we can choose local coordinates
$(z_j,w_j)$ on a neighborhood of $p^{-1}(\overline{U_j})\cap\Omega_\nu$,
with $w_j\in\mathbb{C}$, so that
$p^{-1}(U_j)\cap\Omega_\nu\subset\mathbb{B}^n\times\mathbb{C}$,
$Y\cap p^{-1}(U_j)=\{(z_j,w_j)\in\mathbb{B}^n\times\mathbb{C}\mid w_j=0\}$,
and $p(z_j,w_j)=z_j$.
Since $J$ is finite, there is also an $r>0$ such that, for every $j\in J$,
\begin{equation}\label{eq:smalldischaitteru}
\{(z_j, w_j)\in \mathbb{B}^n\times \mathbb{C}\mid |w_j| < r\}\subset p^{-1}(U_j)\cap \Omega_\nu
\end{equation}
The union of these product neighborhoods over $j\in J$ is a smaller open
neighborhood of $Y$, but it need not be pseudoconvex.
Thus, when considering fiberwise Bergman kernels below, we use the actual
fibers of $\Omega_\nu$, which vary pseudoconvexly, rather than replace them
by these smaller discs.
If the conormal bundle $N_{Y/X}^*$ is unitary flat, we may further arrange
that, whenever $U_j\cap U_k\ne\emptyset$, there is a locally constant
$\mathrm{U}(1)$-valued function $t_{jk}$ on $s(U_j\cap U_k)$ such that
\begin{equation}\label{eq:conormalgaunitaryflat}
dw_j = t_{jk}\,dw_k\qquad\text{on}\ s(U_j\cap U_k)\ (\subset Y)
\end{equation}
Indeed, unitary flatness and the fact that $dw_j$ trivializes $N_{Y/X}^*$
on $s(U_j)$ give nowhere-vanishing holomorphic functions
$f_j\colon s(U_j)\cap Y\to\mathbb{C}^*$ satisfying
\[
\left|\frac{f_j\,dw_j}{f_k\,dw_k}\right|\equiv 1\qquad\text{on}\ s(U_j\cap U_k)\ (\subset Y)
\]
Replacing $w_j$ by $(f_j\circ s\circ p)\cdot w_j$ then gives
\eqref{eq:conormalgaunitaryflat}.

Whenever $Y$ admits a fundamental system of locally pseudoconvex
neighborhoods and $N_{Y/X}^*$ is unitary flat, we adopt the following
conventions:
\begin{itemize}
\item Fix $\nu$ as above and set $\Omega:=\Omega_\nu$;
\item Replace $p$ by $p|_\Omega$, so that $p^{-1}(U_j)$ means
$p^{-1}(U_j)\cap\Omega$ unless otherwise stated;
\item Choose the local coordinates $(z_j,w_j)$ on $p^{-1}(U_j)$ so that
\eqref{eq:smalldischaitteru} holds for some common $r>0$ and all $j\in J$,
and \eqref{eq:conormalgaunitaryflat} holds for locally constant
$\mathrm{U}(1)$-valued functions $t_{jk}$ on $s(U_j\cap U_k)$;
\item Identify $U_j$ with $\mathbb{B}^n$ and each fiber $p^{-1}(z)$ with
a bounded open subset of $\mathbb{C}$.
\end{itemize}
By \eqref{eq:b-ker_zahyouhenkan}, the value of $k_{\Omega_z}(w_j)$ at
$w_j=0$, that is, at $s(z)\in Y$, is independent of the choice of $j\in J$
with $z\in U_j$.
Thus, via $p|_Y\colon Y\cong Z$, we obtain a globally defined function
\begin{equation}\label{eq:zwokomegaz0niokurufunc}
Z\ni z\mapsto k_{\Omega_z}(0) \in \mathbb{R}_{>0}
\end{equation}
on $Z$.
Lemma~\ref{lem:logkgaphninaru} below shows that this function is constant.
Its value changes under a common rescaling of the fiber coordinates,
as follows again from \eqref{eq:b-ker_zahyouhenkan}.
We therefore also normalize the coordinates so that
\begin{equation}\label{eq:0}
\log k_{\Omega_z}(0) \equiv 0
\end{equation}
\end{remark}

\begin{lemma}\label{lem:logkgaphninaru}
The function on $Z$ defined by \eqref{eq:zwokomegaz0niokurufunc} is constant.
\end{lemma}

\begin{proof}
By Theorem~\ref{thm:bergman-log-psh}, the logarithm of this function is
locally plurisubharmonic. Since $Z$ is compact and connected, the maximum
principle implies that it is constant.
\end{proof}

\section{Proofs of the main results}

\subsection{Proof of Theorem \ref{thm:main1}}\label{subsection:maintheoremproof1}

We prove Theorem~\ref{thm:main1}.
The implication $(ii)\implies (i)$ is clear. 
The implication $(iii)\implies(i)$ follows as well. Indeed, if $Y$ admits
a holomorphic tubular neighborhood $\Omega$ in $X$, then
\[
\Omega_\nu:=\Phi^{-1}\left(\left\{\xi \in N_{Y/X} \left|\, |\xi|_h < \frac{1}{\nu+1}\right.\right\}\right)
=\Phi^{-1}\left(\left\{\xi \in N_{Y/X} \left|\, \log |\xi|_h +\log(\nu+1)<0\right.\right\}\right)
\]
defines a fundamental system $\{\Omega_\nu\}_{\nu\in\mathbb{Z}_{\geq 0}}$
as in $(i)$, where $h$ is a flat metric on $N_{Y/X}$ and
$\Phi\colon\Omega\to N_{Y/X}$ is as in
\eqref{eq:holomorphictubularneighborhoodnodefnodoukei}. 
This argument alone does not establish $(iii)\implies(ii)$, since no
relation between $s\circ p$ and $\pi\circ\Phi$ is given a priori,
where $\pi\colon N_{Y/X}\to Y$ is the projection.

We now assume that a system $\{\Omega_\nu\}_{\nu\in\mathbb{Z}_{\geq 0}}$
as in $(i)$ exists, and prove $(ii)$ and $(iii)$.
We use the notation and conventions of Remark~\ref{rmk:daijinarmk}.
The proof has three steps:
\begin{description}
\item[Step 1] For a fixed coordinate chart $U_j$, show that the function
$(z_j,w_j)\mapsto k_{\Omega_{z_j}}(w_j,0)$ is holomorphic on $p^{-1}(U_j)$.
\item[Step 2] On a neighborhood of each $s(U_j)$, construct a new fiber
coordinate $\xi_j$ by integrating $k_{\Omega_{z_j}}(w_j,0)\,dw_j$.
Compute the transition relations near $s(U_j\cap U_k)$ and verify that
the functions $|\xi_j|$ glue to a global function near $Y$.
\item[Step 3] Obtain $(ii)$ by considering
$\textstyle\bigcup_{j\in J}\{|\xi_j|<\ve\}$ for sufficiently small
$\ve>0$, and obtain $(iii)$ from the coordinates $\{(z_j,\xi_j)\}_{j\in J}$.
\end{description}

\subsubsection{Step 1: Analysis in a fixed coordinate neighborhood}

Fix $j\in J$ and consider $p^{-1}(U_j)$.
By the convention in Remark~\ref{rmk:daijinarmk}, this is the part of the
locally pseudoconvex neighborhood $\Omega$ over $U_j\cong\mathbb{B}^n$,
viewed in $\mathbb{B}^n\times\mathbb{C}$.
Since $j$ is fixed, we write $(z,w)$ for $(z_j,w_j)$ and identify $U_j$
with $\mathbb{B}^n$.
We are thus in the setting and notation of
\S\ref{subsection:bergmanker_optimalOT}.
By \eqref{eq:0} and Proposition~\ref{prop:kyoutuketasita}, we have that $k_{\Omega_z}(w, 0)$ is holomorphic. 


\subsubsection{Step 2: Transition relations and integration of $k_{\Omega_z}(w_j,0)\,dw_j$}

Fix $z\in Z$ and an orthonormal basis $\{\eta_\nu\}_{\nu=0}^\infty$ of
$A^2(\Omega_z)$.
For $j,k\in J$ with $z\in U_j\cap U_k$, we compare the expressions for
$k_{\Omega_z}(\bullet,0)$ and the associated holomorphic $1$-forms in the
coordinates $w_j$ and $w_k$ on this fixed fiber.
Write
\[
\eta_\nu = f_{\nu,j}(w_j)\,dw_j = f_{\nu,k}(w_k)\,dw_k 
\]
Then
\begin{align*}
k_{\Omega_z}(w_j, w_j') &= \sum_{\nu=0}^\infty f_{\nu,j}(w_j)\cdot \overline{f_{\nu,j}(w_j')},\\
k_{\Omega_z}(w_k, w_k') &= \sum_{\nu=0}^\infty f_{\nu,k}(w_k)\cdot \overline{f_{\nu,k}(w_k')}
\end{align*}
and hence
\begin{align*}
k_{\Omega_z}(w_j, 0)\,dw_j &= \left(\sum_{\nu=0}^\infty f_{\nu,j}(w_j)\cdot \overline{f_{\nu,j}(0)}\right)\,dw_j
= \sum_{\nu=0}^\infty \overline{f_{\nu,j}(0)}\cdot \eta_\nu,\\
k_{\Omega_z}(w_k, 0)\,dw_k &= \left(\sum_{\nu=0}^\infty f_{\nu,k}(w_k)\cdot \overline{f_{\nu,k}(0)}\right)\,dw_k
= \sum_{\nu=0}^\infty \overline{f_{\nu,k}(0)}\cdot \eta_\nu
\end{align*}
Since $f_{\nu,j}(w_j)\,dw_j=f_{\nu,k}(w_k)\,dw_k$, we have
\[
f_{\nu,j}(0)\,dw_j = f_{\nu,k}(0)\,dw_k \qquad \text{on}\ s(U_j\cap U_k)
\]
By \eqref{eq:conormalgaunitaryflat}, it follows that
\[
f_{\nu,j}(0)\cdot t_{jk} = f_{\nu,k}(0)
\]
We therefore obtain
\begin{equation}\label{eq:komagazdwjnojkdenohenkangatjk}
k_{\Omega_z}(w_k, 0)\,dw_k = \overline{t_{jk}}\, k_{\Omega_z}(w_j, 0)\,dw_j
\end{equation}

We now allow the base point $z$ to vary.
Recall that the coordinates $(z_j,w_j)$ were chosen so that
\eqref{eq:smalldischaitteru} holds for a common $r>0$ and all $j\in J$.
For each coordinate chart $U_j$, define a function
\[
\xi_j \colon \{(z_j, w_j)\in U_j\times \mathbb{C} \mid |w_j|<r\} \to \mathbb{C}
\]
as follows. For $z_j\in U_j$ and $x=(z_j,a)$ with $|a|<r$, let $\Gamma_x$
be a path in the fiber disc
$\{(z_j,w_j)\in U_j\times\mathbb{C}\mid |w_j|<r\}$ from $w_j=0$ to
$w_j=a$, and set
\begin{equation}\label{eq:komegaw0dwwosekibunsitexiwotukurutoko}
\xi_j(x) := \int_{\Gamma_x}k_{\Omega_z}(w_j, 0)\,dw_j
\end{equation}
By Step~1, the integrand is holomorphic. Since the integration is over a
disc, the value is independent of the choice of $\Gamma_x$, and $\xi_j$
is holomorphic.
Moreover, \eqref{eq:komagazdwjnojkdenohenkangatjk} gives
$\xi_k=\overline{t_{jk}}\,\xi_j$.
As $t_{jk}$ takes values in $\mathrm{U}(1)$, this is equivalent to
\begin{equation}\label{eq:xinohenkan}
\xi_j = t_{jk}\,\xi_k
\end{equation}
on the intersection
\[
\{(z_j, w_j)\in U_j\times \mathbb{C} \mid |w_j|<r\} \cap \{(z_k, w_k)\in U_k\times \mathbb{C} \mid |w_k|<r\}.
\]
Here and below, we write $t_{jk}$ also for the pullback $t_{jk}\circ s\circ p$
of $t_{jk}\colon s(U_j\cap U_k)\to\mathrm{U}(1)$.

\subsubsection{Step 3: Completion of the proof}

The construction of $\xi_j$ and the normalization \eqref{eq:0} give
\[
\xi_j\circ s \equiv 0,\qquad d\xi_j = k_{\Omega_z}(0)\,dw_j = dw_j \quad \text{on}\ s(U_j)
\]
Thus there is a neighborhood $V_j$ of $s(U_j)$ in $p^{-1}(U_j)$ on which
$\xi_j$ is a defining function for $Y\cap V_j$ and $(z_j,\xi_j)$ forms a
system of local coordinates.

Together with \eqref{eq:xinohenkan}, the neighborhoods
$\textstyle\bigcup_{j\in J}\{|\xi_j|<\ve\}$, for $0<\ve\ll1$, give the assertion $(ii)$.
By \eqref{eq:conormalgaunitaryflat}, \eqref{eq:xinohenkan}, and
$p(z_j,w_j)=z_j$, the maps
\[
V_j\ni (z_j, \xi_j)\mapsto \xi_j\cdot \left.\frac{\del}{\del w_j}\right|_{s(z_j)} \in (\text{The total space of}\ N_{Y/X})
\]
glue to a holomorphic open embedding of $\textstyle\bigcup_{j\in J}V_j$
onto a neighborhood of the zero section in the total space of $N_{Y/X}$, which show the assertion $(iii)$.
This proves Theorem~\ref{thm:main1}. \hfill $\Box$

\subsection{Proof of Theorem \ref{thm:main2}}\label{subsection:maintheoremproof2}

Assume condition $(ii)$, which in particular implies $(i)$.
For every sufficiently large $\nu$, we are in the setting of
Remark~\ref{rmk:daijinarmk}, so the preceding argument applies with
$\Omega=\Omega_\nu$.
Repeat that argument through the construction
\eqref{eq:komegaw0dwwosekibunsitexiwotukurutoko}.
Since each fiber $\Omega_z$ is connected and simply connected, the
functions $\xi_j$ defined by this integral are now well defined on all of
$p^{-1}(U_j)$, that is, on $p^{-1}(U_j)\cap\Omega$ in the original notation.
By Example~\ref{ex:bergman-simply-connected} and \eqref{eq:0},
each $\xi_j$ maps each fiber over $U_j$ biholomorphically onto
$\{\xi\in\mathbb{C}\mid |\xi|<1/\sqrt{\pi}\}$.
Using the coordinates $(z_j,\xi_j)$ as in Step~3 above gives the assertion.
\hfill $\Box$

\section{Application}\label{section:localautsusp}

As explained in \S\ref{section:introduction}, a main motivation for this
work is Ueda's example \cite[\S~5.4]{U}.
It is constructed from a germ $f\in\mathcal{O}_{\mathbb{C},0}$ of a
holomorphic function in one variable of the form
\begin{equation}\label{eq:ftteiukansuu}
f(w)=\lambda w + \sum_{n=2}^\infty a_n\,w^n \in \mathcal{O}_{\mathbb{C}, 0}\qquad (\lambda\in \mathbb{C}^*,\ a_2, a_3,\dots\in\mathbb{C}). 
\end{equation}
More precisely, the example is obtained by \emph{suspension} over an
elliptic curve, with monodromy in one direction given by the prescribed
germ $f$. Following \cite{K}, we describe a slightly more general version
of this construction and discuss the consequences of our main results.

Let $f$ be as in \eqref{eq:ftteiukansuu}, and let
$Z=\mathbb{C}^*/\sim_{\widetilde{\mu}}$ be a smooth elliptic curve, where
$\widetilde{\mu}$ is a constant such that $\mu:=|\widetilde{\mu}|$ satisfies
$0<\mu<1$, and $\sim_{\widetilde{\mu}}$ is the equivalence relation on
$\mathbb{C}^*:=\mathbb{C}\setminus\{0\}$ generated by
$z\sim_{\widetilde{\mu}}\widetilde{\mu}\cdot z$.
We denote by $q\colon\mathbb{C}^*\to Z$ the natural quotient map.
As shown in \cite[\S~2]{K}, there exists a smooth complex surface $X$
with a holomorphic surjective submersion $p\colon X\to Z$ and a
holomorphic section $s\colon Z\to X$ satisfying the following conditions:
\begin{itemize}
\item There is a non-singular holomorphic foliation $\mathcal{F}$ on a
neighborhood of $Y:=s(Z)$ in $X$, with $Y$ as a leaf;
\item The holonomy morphism
${\rm Hol}_{\mathcal{F},Y}\colon\pi_1(Y,*)\to\mathcal{O}_{\mathbb{C},0}$
of $\mathcal{F}$ along $Y$ satisfies
\[
{\rm Hol}_{\mathcal{F}, Y}(\gamma_1)(w)=w,\ {\rm Hol}_{\mathcal{F}, Y}(\gamma_2)(w)=f(w), 
\]
where $\gamma_1:=[t\mapsto q(e^{2\pi\sqrt{-1}t})]$ and
$\gamma_2:=[t\mapsto q(e^{-tL})]$ in $\pi_1(Z,q(1))$,
with $0\leq t\leq1$ and $e^L=\widetilde{\mu}$.
\end{itemize}
Moreover, the germ of the neighborhood of $Y$ in $X$ is uniquely
determined by the germ $f$ in \eqref{eq:ftteiukansuu} and the constant
$\widetilde{\mu}\in\mathbb{C}$ with $0<|\widetilde{\mu}|<1$
\cite[\S~2]{K}.

The results of \cite{K,KO}, together with Ueda's classification \cite{U},
can be summarized as follows.
\begin{itemize}
\item If $\lambda\not\in\mathrm{U}(1)$, if $\lambda\in\mathrm{U}(1)$ is
 torsion and some iterate of $f$ is the identity, or if $0$ lies in the
 interior of a Siegel disc of $f$, then $Y$ admits a pseudoflat fundamental
 system of neighborhoods \cite[Theorem~1.2~(ii)]{KO}. These are the cases in classes $(\beta')$ and
 $(\beta'')$ of Ueda's classification.
\item If $\lambda\in\mathrm{U}(1)$ is torsion and no iterate of $f$ is the
 identity, then $Y$ admits a fundamental system of neighborhoods in $X$
 with strongly pseudoconcave boundaries \cite[Theorem 1 and its corollary]{U} (see also \cite[Theorem~1.1~(ii)]{K} and \cite[Theorem~1.2~(i)]{KO}). This is class $(\alpha)$ in Ueda's
 classification.
\item Suppose that $\lambda\in\mathrm{U}(1)$ is non-torsion and $f$ is not
 linearizable near $0$. This generalizes Ueda's example
 \cite[\S~5.4]{U}, and $(Y,X)$ belongs to class $(\gamma)$
 \cite[Theorems~1.1 and~1.2]{KO}.
 In this case, for every open neighborhood $\Omega\subset X$ of $Y$ and
 every continuous function $\Phi\colon\Omega\to(-\infty,+\infty]$ such
 that $\Phi|_{\Omega\setminus Y}$ is plurisubharmonic, $\Phi$ is bounded
 above near $Y$ \cite[Theorem~1.5]{KO}
 (see also \cite[Theorem~1.1~(iii)]{K} for Ueda's example).
 Here $(-\infty,+\infty]$ is equipped with the topology homeomorphic to
 that of the interval $(0,1]$.
\end{itemize}

In the present setting, Ueda's classification \cite{U} compares the
line bundle defined by the divisor $Y$ with the unitary flat extension
of $N_{Y/X}$ to a neighborhood of $Y$.
In classes $(\beta')$ and $(\beta'')$, these bundles are holomorphically
isomorphic on a neighborhood of $Y$; in class $(\alpha)$, they fail to
be isomorphic on some finite-order infinitesimal neighborhood of $Y$;
and in class $(\gamma)$, they are formally isomorphic along $Y$ but
are not holomorphically isomorphic on any neighborhood of $Y$.

The example constructed by Ueda in \cite[\S~5.4]{U} was the first
example of class $(\gamma)$.
In the elliptic-curve case, it is an example $p\colon X\to Z$ in the
above setting, with a holomorphic section $Y=s(Z)$ and $n=1$, such that,
for every positive integer $m_0$ and every sufficiently small
neighborhood $\Omega$ of $Y$ in $X$, there exist an integer $m>m_0$
and a compact non-singular complex curve
$\Gamma\subset\Omega\setminus Y$ homologous to $mY$ in $X$.

The results recalled above impose restrictions on plurisubharmonic
functions near $Y$ in class $(\gamma)$, including Ueda's example.
However, they do not by themselves determine whether $Y$ admits a
fundamental system of locally pseudoconvex neighborhoods.
Our main theorem yields the following.

\begin{corollary}\label{cor:main}
In the above construction, the following are equivalent:
\begin{enumerate}
\item[$(i)$] The germ $f$ is holomorphically linearizable at $0$,
that is, there is a local holomorphic change of coordinate
$u=\psi(w)$, with $\psi(0)=0$ and $\psi'(0)\ne0$, such that
\[
(\psi\circ f\circ\psi^{-1})(u)=\lambda u
\]
for $u$ sufficiently close to $0$.
\item[$(ii)$] The holomorphic line bundle $[Y]$ associated with the
divisor $Y$ admits a $C^\infty$ Hermitian metric whose Chern curvature
is semi-positive at every point of $X$.
\item[$(iii)$] $Y$ admits a holomorphic tubular neighborhood in $X$.
\item[$(iv)$] $Y$ admits a fundamental system of locally pseudoconvex
neighborhoods in $X$.
\end{enumerate}
In particular, in Ueda's example \cite[\S~5.4]{U}, $Y$ does not admit
a fundamental system of locally pseudoconvex neighborhoods in $X$.
\end{corollary}

\begin{proof}
By \cite[Theorem~1.2~(ii) and Corollary~1.6]{KO} and the
suspension construction \cite[\S~2]{K}, conditions $(i)$--$(iii)$
are equivalent to $(Y,X)$ being of type $(\beta')$ or $(\beta'')$.
Since $N_{Y/X}$ has degree zero on the elliptic curve $Y$ and is
therefore unitary flat, $(iii)$ and $(iv)$ are equivalent by
Theorem~\ref{thm:main1}.
The last assertion follows since Ueda's example
\cite[\S~5.4]{U} is of type $(\gamma)$.
\end{proof}




\begin{thebibliography}{HMY}
\bibitem[Be]{Be} \textsc{B. Berndtsson}, Subharmonicity properties of the Bergman kernel and some other functions associated to pseudoconvex domains, Ann. Inst. Fourier (Grenoble) {\bf 56}, 6 (2006), 1633--1662.
\bibitem[BL]{BL} \textsc{B. Berndtsson, L. Lempert}, A proof of the Ohsawa--Takegoshi theorem with sharp estimates, J. Math. Soc. Japan {\bf 68}, 4 (2016), 1461--1472.
\bibitem[Bl]{Bl} \textsc{Z. B{\l}ocki}, Suita conjecture and the Ohsawa--Takegoshi extension theorem, Invent. Math. {\bf 193} (2013), 149--158.
\bibitem[GZ]{GZ} \textsc{Q. Guan, X. Zhou},
A solution of an $L^2$ extension problem with an optimal estimate and applications,
Ann. of Math. {\bf 181}, 3 (2015), 1139--1208.
\bibitem[HMY]{HMY} \textsc{S. Hamano, F. Maitani, H. Yamaguchi},
Variation formulas for principal functions, II: Applications to variation for harmonic spans,
Nagoya Math. J. {\bf 204} (2011), 19--56.
\bibitem[K]{K} \textsc{T. Koike}, Plurisubharmonic functions on a neighborhood of a torus leaf of a certain class of foliations, Forum Math. {\bf 31}, 6 (2019), 1457--1466. 
\bibitem[KO]{KO} \textsc{T. Koike, N. Ogawa}, On the neighborhood of a torus leaf and dynamics of holomorphic foliations, Tohoku Math. J. (2) {\bf 77}, 3 (2025), 421--443. 
\bibitem[MY]{MY} \textsc{F. Maitani, H. Yamaguchi}, Variation of Bergman metrics on Riemann surfaces, Math. Ann. {\bf 330} (2004), 477--489.
\bibitem[N]{N} \textsc{T. Nishino}, Nouvelles recherches sur les fonctions enti\`eres de plusieurs variables complexes (II). Fonctions enti\`eres qui se r\'eduisent \`a celles d'une variable, J. Math. Kyoto Univ. {\bf 9} (1969), 221--274.
\bibitem[O]{O} \textsc{T. Ohsawa}, $L^2$ Approaches in Several Complex Variables: Towards the Oka--Cartan Theory with Precise Bounds, second edition, Springer Monographs in Mathematics, Springer, Tokyo, 2018.
\bibitem[U]{U} \textsc{T. Ueda},  On the neighborhood of a compact complex curve with topologically trivial normal bundle, J. Math. Kyoto Univ. {\bf 22} (1982/1983), 583--607. 
\bibitem[Z]{Z} \textsc{S. Zakeri},
Conformal fitness and uniformization of holomorphically moving disks,
Trans. Amer. Math. Soc. {\bf 368}, 2 (2016), 1023--1049.
\end{thebibliography}
\end{document}